\documentclass[a4paper, 11pt]{amsart}
\usepackage{dsfont}
\usepackage{amssymb}
\usepackage{amsfonts}
\usepackage{amsthm}
\usepackage{pdfsync}
\usepackage{amsmath}
\usepackage{verbatim}
\usepackage{a4wide}

\newtheorem{lemma}{Lemma}[section]
\newtheorem{theorem}{Theorem}[section]
\newtheorem{proposition}[lemma]{Proposition}

\newtheorem{corollary}[lemma]{Corollary}
\theoremstyle{definition}

\theoremstyle{definition}
\newtheorem{remark}[lemma]{Remark}
\theoremstyle{definition}

{\catcode `\@=11 \global\let\AddToReset=\@addtoreset}
\AddToReset{equation}{section}

\newcommand{\A}{\mathfrak{H}}

\newcommand{\D}{\mathrm{d}}

\newcommand{\Q}{\mathcal{E}_{\rm pot}}
\newcommand{\s}{\mathcal{S}}

\newcommand{\R}{\mathbb{R}}
\newcommand{\N}{\mathbb{N}}
\newcommand{\EH}{\mathcal E_H}
\newcommand{\FT}{\mathcal F_T}
\newcommand{\e}{{\varepsilon }}

\newcommand{\norm}[1]{\| #1 \|}
\def\tr{\mathop{\rm tr}\nolimits}

\begin{document}

\title[Thermal effects in gravitational Hartree systems]
{Thermal effects in gravitational Hartree systems}

\author[G.L.~Aki] {Gonca L.~Aki}

\author[J.~Dolbeault] {Jean Dolbeault}

\author[C.~Sparber]{Christof Sparber}

\address[G.L.~Aki]
{Faculty of Mathematics, University of Vienna, Nordbergstra\ss e 15, A-1090 Vienna, Austria}\email{gonca.aki@univie.ac.at}

\address[J.~Dolbeault]
{Ceremade (UMR CNRS no.~7534), Universit\'e Paris-Dauphine, Place de Lattre de Tassigny, F-75775 Paris C\'edex 16, France}\email{dolbeaul@ceremade.dauphine.fr}

\address[C.~Sparber]
{Department of Mathematics, Statistics, and Computer Science, M/C 249, University of Illinois at Chicago, 851 S. Morgan Street, Chicago, IL 60607, USA}\email{sparber@math.uic.edu}

\date{\today}


%
%
%
%


\thanks{This publication has been supported by Award No.~KUK-I1-007-43 of the King Abdullah University of Science and Technology (KAUST). J.~Dolbeault and C.~Sparber have been supported, respectively, by the ANR-08-BLAN-0333-01 project CBDif-Fr and by the University research fellowship of the Royal Society. G.L.~Aki acknowledges the support of the FWF, grant no.~W 800-N05 and funding by WWTF project (MA45).}

\begin{abstract}
We consider the non-relativistic Hartree model in the gravitational case, i.e.~with attractive Coulomb-Newton interaction. For a given mass $M>0$, we construct stationary states with non-zero temperature $T$ by minimizing the corresponding free energy functional. It is proved that minimizers exist if and only if the temperature of the system is below a certain threshold $T^*>0$ (possibly infinite), which itself depends on the specific choice of the entropy functional. 
We also investigate whether the corresponding minimizers are mixed or pure quantum states and characterize a critical temperature $T_c \in (0, T^*)$ above which mixed states appear.
\end{abstract}

\maketitle
\thispagestyle{empty}

\section{Introduction}\label{sint}


In this paper we investigate the \emph{non-relativistic gravitational Hartree system with temperature}. This model can be seen as a mean-field description of a system of self-gravitating quantum particles. It is used in astrophysics to describe 
so-called \emph{Boson stars}. In the present work, we are particularly interested in \emph{thermal effects}, i.e.~(qualitative) differences to the zero temperature case.

A physical state of the system will be represented by a density matrix operator $\rho\in\mathfrak S_1(L^2(\R^3))$, i.e.~a positive self-adjoint trace class operator acting on $L^2(\R^3;\mathbb C)$. 
Such an operator $\rho$ can be decomposed as
\begin{equation}\label{decom}
\rho =\sum_{j\in\N}\lambda_j\,|\psi_j\rangle\langle \psi_j|
\end{equation}
with an associated sequence of eigenvalues $(\lambda_j)_{j\in\N}\in\ell^1$, $\lambda_ j \ge 0$, usually called \emph{occupation numbers}, and a corresponding sequence of eigenfunction $(\psi_j)_{j\in\N}$, 
forming a complete orthonormal basis of $L^2(\R^3)$, cf.~\cite{Si}. By evaluating the kernel $\rho(x,y)$ on its diagonal, we obtain the corresponding particle density
\[
n_\rho(x)=\sum_{j\in\N}\lambda_j\,|\psi_j(x)|^2\in L_+^1(\R^3)\;.
\]
In the following we shall assume that
\begin{equation}\label{mass}
\int_{\R^3} n_\rho(x) \, \D x = M\;,
\end{equation}
for a given total mass $M>0$. We assume that the particles interact solely via gravitational forces. The corresponding \emph{Hartree energy} of the system is then given by
\begin{align*}
\EH[\rho] : = \mathcal E_{\rm kin}[\rho] - \mathcal E_{\rm pot}[\rho]= \tr(-\Delta\,\rho)- \frac{1}{2} \tr(V_\rho\,\rho)\;,
\end{align*}
where $V_\rho$ denotes the \emph{self-consistent potential}
\[
V_\rho = n_\rho \ast \frac{1}{ |\, \cdot \, |}
\]
and $`\ast$' is the usual convolution w.r.t.~$x\in\R^3$. Using the decomposition \eqref{decom} for $\rho$, the Hartree energy can be rewritten as
\[
\EH[\rho] = \sum_{j\in\N}\lambda_j\int_{\R^3}|\nabla \psi_j(x)|^2\,\D x-\frac{1}{2}\iint_{\R^3\times\R^3}\frac{n_\rho(x)\,n_\rho(y)}{|x-y|}\;\D x\,\D y\;.
\]
To take into account thermal effects, we consider the associated \emph{free energy functional} 
\begin{equation}\label{freeenergy}
\FT[\rho]
:=\EH[\rho]-T\,\s[\rho]
\end{equation}
where $T\ge 0$ denotes the temperature and $\mathcal S[\rho]$ is the \emph{entropy functional}
\[
\s[\rho]:=-\tr\beta(\rho)\;.
\]
The \emph{entropy generating function} $\beta$ is assumed to be convex, of class $C^1$ and will satisfy some additional properties to be prescribed later on. 
The purpose of this paper is to investigate the existence of \emph{minimizers} for $\FT$ with fixed mass $M>0$ and temperature $T\ge 0$ and study their qualitative properties. 
These minimizers, often called \emph{ground states}, can be interpreted as stationary states for the time-dependent system
\begin{equation}\label{equationrho}
i\,\frac{\D}{\D t} \rho(t) = \, [H_{\rho(t)}, \rho(t)]\;, \quad \rho(0)=\rho_{\rm in}\;.
\end{equation}
Here $[A, B]=A\,B-B\,A$ denotes the usual commutator and $H_\rho$ is the mean-field \emph{Hamiltonian operator}
\begin{equation}\label{Hamil}
H_{\rho}:=-\Delta- n_{\rho} \ast \frac{1}{|\cdot|}\;.
\end{equation}
Using again the decomposition \eqref{decom}, this can equivalently be rewritten as a system of (at most) countably many Schr\"odinger equations coupled through the mean field potential $V_\rho$:
\begin{equation}\label{equationpsi}
\left\{\begin{array}{l}
i\,\partial_t\psi_j +\Delta\,\psi_j +V(t,x)\,\psi_j = 0\;,\quad j\in\N\;,\\[6pt]
- \Delta V_\rho\,= \, 4 \pi \sum_ {j\in\N}\lambda_j\,|\psi_j(t,x)|^2\,.
\end{array}\right.
\end{equation}
This system is a generalization of the gravitational Hartree equation (also known as the \emph{Schr\"odinger-Newton model}, see \cite{StVu}) to the case of mixed states. 
Notice that it reduces to a finite system as soon as only a finite number of $\lambda_j$ are non-zero. In such a case, $\rho$ is a finite rank operator.

\smallskip Establishing the existence of stationary solutions to nonlinear Schr\"odinger models by means of variational methods is a classical idea, cf.~for instance~\cite{Lieb-Loss}. A particular advantage of such an approach is that in most cases one can directly deduce \emph{orbital stability} of the stationary solution w.r.t.~the dynamics of \eqref{equationrho} or, equivalently, \eqref{equationpsi}. In the case of \emph{repulsive} self-consistent interactions, describing e.g.~electrons, this has been successfully carried out in \cite{DoFeLe, DoFeLoPa, DoFeMa, MaReWo}. In addition, existence of stationary solutions in the repulsive case has been obtained in \cite{Ma, Ni1, Ni2, Ni3} using convexity properties of the corresponding energy functional.

In sharp contrast to the repulsive case, the gravitational Hartree system of stellar dynamics, does \emph{not} admit a convex energy and thus a more detailed study of minimizing sequences is required. To this end, we first note that at zero temperature, i.e.~$T=0$, the free energy $\FT[\rho]$ reduces to the gravitational Hartree energy $\EH[\rho]$. For this model, existence of the corresponding zero temperature ground states has been studied in \cite{Lie, Li2, Li3} and, more recently, in \cite{StVu}. 
Most of these works rely on the so-called \emph{concentration-compactness method} introduced by Lions in \cite{Li}. According to \cite{Lie}, it is known that for $T=0$ the minimum of the Hartree energy is uniquely achieved by an appropriately normalized \emph{pure state}, i.e.~a rank one density matrix $\rho_0 = M\,|\psi_0\rangle\langle \psi_0|$. The concentration-compactness method has later been adapted to the setting of density matrices, see for instance \cite{LeLe} for a recent paper written this framework, in which the authors study a \emph{semi-relativistic} model of Hartree-Fock type at zero temperature.

\begin{remark} In the classical kinetic theory of self-gravitating systems, a variational approach based on the so-called \emph{Casimir functionals} has been repeatedly used to prove existence and orbital stability of stationary states of relativistic and non-relativistic Vlasov-Poisson models: see for instance~\cite {Wa1, Wa2, Wo, Ge1, Ge2, Sc, DoSaSo, Ge3, SaSo}. These functionals can be regarded as the classical counterpart of $\FT[\rho]$ and such an analogy between classical and quantum mechanics has already been used in \cite{MaReWo, DoFeLoPa, DoFeMa, DoFeLe}.
\end{remark} 

In view of the quoted results, the purpose of this paper can be summarized as follows: 
First, we shall prove the existence of minimizers for $\FT$, extending the results of \cite{Lie, Li2, Li3, StVu} to the case of non-zero temperature. As we shall see, a \emph{threshold in temperature} arises due to 
the competition between the Hartree energy and the entropy term and we find that minimizers of $\FT$ exist only \emph{below a certain maximal temperature} $T^*>0$, 
which depends on the specific form of the entropy generating function $\beta$. One should note that, by using the scaling properties of the system, the notion of a maximal temperature for a given mass $M$ can be rephrased into a corresponding threshold for the mass at a given, fixed temperature $T$. Such a critical mass, however, has to be clearly distinguished from the well-known \emph{Chandrasekhar mass} threshold in semi-relativistic models, cf.~\cite{LiYau, Len, LeLe}. Moreover, depending on the choice of $\beta$, it could happen that $T^*= +\infty$, in which case minimizers of $\FT$ would exist even if the temperature is taken arbitrarily large. In a second step, we shall also 
study the qualitative properties of the ground states with respect to the temperature $T\in [0, T^*)$. In particular, we will prove that there exists a certain \emph{critical temperature} $T_c >0$, 
above which minimizers correspond to \emph{mixed quantum states}, i.e.~density matrix operators with rank higher than one. If $T<T_c$, minimizers are pure states, as in the zero temperature model.

\medskip In order to make these statements mathematically 
precise, we introduce 
\[
\A:=\Big\{\rho: L^2(\R^3)\rightarrow L^2(\R^3)\ : \ \rho\ge 0\,,~\rho\in\mathfrak S_1\,,~\sqrt{-\Delta}\,\rho\,\sqrt{-\Delta}\in\mathfrak S_1\Big\}
\]
and consider the norm
\[
\norm{\rho}_{\A}:=\tr\rho+\tr\big(\sqrt{-\Delta}\,\rho\,\sqrt{-\Delta}\big)\;.
\]
The set $\A$ can be interpreted as the cone of nonnegative density matrix operators with finite energy. Using the decomposition \eqref{decom}, if $\rho\in\A$, we obtain that $\psi_j\in H^1(\R^3)$ for all $j\in\N$ such that $\lambda_j>0$. Taking into account the mass constraint \eqref{mass} we define the set of physical states by
\[
\begin{array}{ll}
\A_M:=\{\rho\in\A \ : \ \tr\rho=M\}\;.
\end{array}
\]
We denote the infimum of the free energy functional $\FT$, defined in \eqref{freeenergy}, by
\begin{equation}\label{infimum}
i_{M,T}=\inf_{\rho\in\A_M}\FT[\rho]\;.
\end{equation}
The set of minimizers will be denoted by $\mathfrak M_M \subset \A_M$.
As we shall see in the next section, $i_{M,T} <0$ if $\mathfrak M_M\neq\emptyset$. This however can be guaranteed only below a certain maximal temperature $T^* = T^*(M)$ given by
\begin{equation}\label{Tstar}
T^*(M):=\sup\{T>0~:~i_{M,T}<0\}\;.
\end{equation}
This maximal temperature $T^*$ will depend on the choice of the entropy generating function $\beta$ for which we impose the following assumptions:
\begin{itemize}
\item[($\beta$1)] ${\beta}$ is strictly convex and of class $C^1$ on $[0,\infty)$,\vspace*{6pt}
\item[($\beta$2)] ${\beta}\ge 0$ on $[0,1]$ and $\beta(0)=\beta'(0)=0$,\vspace*{6pt}
\item[($\beta$3)] $\sup_{m\in(0,\infty)}\frac{m\,\beta'(m)}{\beta(m)}\le3$.
\end{itemize}
A typical example for the function $\beta$ reads
\[\label{betaexam}
\beta(s)= s^p\,, \quad p\in(1,3]\;.
\]
Such a power law nonlinearity is of common use in the classical kinetic theory of self-gravitating systems known as \emph{polytropic gases}. One of the main features of such models is to give rise to orbitally stable stationary states with 
\emph{compact support}, cf.~\cite{GuRe, Ge2, Ge3, Wa1, Wa2, Wo}, clearly a desirable feature when modeling stars. 
We shall prove in Section 6, that $T^*$ is \emph{finite} if $p$ is not too large. The limiting case as $p$ approaches $1$ corresponds to $\beta(s) = s\,\ln s$ 
but in that case the free energy functional is \emph{not} bounded from below, see \cite{MR992653} for a discussion in the Coulomb repulsive case, which can easily be adapted to our setting.

Up to now, we have made no distinction between \emph{pure states}, corresponding density matrix operators with rank one, and \emph{mixed states}, corresponding to operators with finite or infinite rank. 
In~\cite{Lie} Lieb has proved that for $T=0$ minimizers are pure states. As we shall see, this is also the case when~$T$ is positive but small and as a consequence we have: $i_{M,T}=i_{M,0}+T\,\beta(M)$. 
Let us define
\begin{equation}\label{criticalT}
T_c(M):=\max\big\{T>0\;: \;i_{M,T}=i_{M,0}+\tau\,\beta(M)\;\forall\;\tau\in(0,T]\,\big\}\;.
\end{equation}
With these definitions in hand, we are now in the position to state our main result.
\begin{theorem} \label{theo} Let $M>0$ and assume that ($\beta$1)--($\beta$3) hold. Then, the maximal temperature~$T^*$ defined in \eqref{Tstar} is positive, possibly infinite, and the following properties hold:
\begin{itemize}
\item[(i)] For all $T< T^*$, there exists a density operator $\rho\in\A_M$ such that $\FT[\rho]=i_{M,T}$. Moreover $\rho$ solves the self-consistent equation
\[
\rho=(\beta')^{-1}\big((\mu -H_\rho)/T\big)
\]
where $H_\rho$ is the mean-field Hamiltonian defined in \eqref{Hamil} and $\mu<0$ denotes the Lagrange multiplier associated to the mass constraint.
\item[(ii)] The set of all minimizers $\mathfrak M_M\subset \A_M$ is orbitally stable under the dynamics of \eqref{equationrho}.
\item[(iii)] The critical temperature $T_c$ defined in \eqref{criticalT} is finite and a minimizer $\rho\in\mathfrak M_M$ is a pure state if and only if $T\in[0,T_c]$.
\item[(iv)] If, in addition, $\beta(s) = s^p$ with $p\in (1,7/5)$, then $T^*<+\infty$.
\end{itemize}
\end{theorem}
The proof of this theorem will be a consequence of several more detailed results. We shall mostly rely on the concentration-compactness method, adapted to the framework of trace class operators. Our approach is therefore similar to the one of \cite{DoFeLe} and \cite{LeLe}, with differences due, respectively, to the sign of the interaction potential and to non-zero temperature effects. Uniqueness of minimizers (up to translations and rotations) 
is an open question for $T> T_c$. For $T\in[0,T_c]$, the problem is reduced to the pure state case, for which uniqueness has been proved in \cite{Lie} (also see \cite{Len2}).

\medskip This paper is organized as follows: In Section \ref{sec:basic} we collect several basic properties of the free energy. In particular we establish the existence of a maximal temperature $T^*>0$ and derive the 
self-consistent equation for $\rho \in \mathfrak H_M$. 
In Section \ref{sec:a-priori}, we derive an important a priori inequality for minimizers, the so-called \emph{binding inequality}, which is henceforth used in proving the existence of minimizers in Section \ref{sec:exist}. 
Having done that, we shall prove in Section \ref {sec:critical} that minimizers are mixed states for $T>T_c$, and we shall also characterize $T_c$ in terms of the eigenvalue problem associated to the case $T=0$. 
In Section~\ref{sec:finite}, we shall prove that $T^*$ is indeed finite in the polytropic case, provided $p<7/5$ and furthermore establish some qualitative properties of the minimizers as $T \to T^*< +\infty$. Finally, Section~\ref{conclusion} is devoted to some remarks on the sign of the Lagrange multiplier associated to the mass constraint and related open questions.

\section{Basic properties of the free energy}\label{sec:basic}

\subsection{Boundedness from below and splitting property}

As a preliminary step, we observe that the functional $\FT$ introduced in \eqref{freeenergy} is well defined and $i_{M,T}>-\infty$.
\begin{lemma}\label{Lem:Bound} Assume that ($\beta$1)--($\beta$2) hold. The free energy $\FT$ is well-defined on $\A_M$ and $i_{M,T}$ is bounded from below. If $\FT[\rho]$ is finite, then $\sqrt{n_\rho}$ is bounded in $H^1(\R^3)$.\end{lemma}
\begin{proof} In order to establish a bound from below, we shall first show that the potential energy $\Q[\rho]$ can be bounded in terms of the kinetic energy. To this end, note that for every $\rho\in\A$ we have
\[
\Q[\rho]\leq C\,\norm{n_\rho}_{L^1}^{3/2}\,\norm{n_\rho}_{L^3}^{1/2}
\]
by the Hardy-Littlewood-Sobolev inequality. Next, by Sobolev's embedding, we know that $\norm{n_\rho}_{L^3}$ is controlled by $\norm{\nabla\sqrt{n_\rho}}_{L^2}^2$ which, using the decomposition \eqref{decom}, is bounded by $\tr(-\Delta\,\rho)$. Hence we can conclude that
\begin{equation}\label{potentialestimate}
\Q[\rho]\leq C\,\norm{n_\rho}_{L^1}^{3/2}\,\tr(-\Delta\,\rho)^{1/2}
\end{equation}
for some generic positive constant $C$. By conservation of mass, the free energy is therefore bounded from below on $\A_M$ according to
\[
\FT[\rho]\geq \tr(-\Delta\,\rho)-C\,M^{3/2}~\tr(-\Delta\,\rho)^{1/2}\geq -\frac14~C^2 \,M^3
\]
uniformly w.r.t.~$\rho\in\mathcal H_M$, thus establishing a lower bound on $i_{M,T}$. For the entropy term $\s[\rho]=- \tr\beta(\rho)$ we observe that, since $\beta$ is convex and $\beta(0)=0$, it holds $0\leq\beta(\rho)\leq\beta(M)\,\rho$ for all $\rho\in\A$ and $\beta(\rho)\in\mathfrak S_1$, provided $\rho\in\mathfrak S_1$. Hence, all quantities involved in the definition of $\FT$ are well-defined and bounded on $\A_M$. \end{proof}

Throughout this work, we shall use smooth \emph{cut-off functions} defined as follows. Let $\chi$ be a fixed smooth function on $\R^3$ with values in $[0,1]$ such that, for any $x\in\R^3$, $\chi(x)=1$ if $|x|<1$ and $\chi(x)=0$ if $|x|\geq2$. For any $R>0$, we define $\chi_R$ and $\xi_R$ by
\begin{equation}\label{cutoff}
\chi_R(x)=\chi(x/R)\quad\mbox{and}\quad\xi_R(x)=\sqrt{1- \chi(x/R)^2}\quad\forall\;x\in\R^3\,.
\end{equation}
The motivation for introducing such cut-off functions is that, for any $u\in H^1(\R^3)$ and any potential~$V$, we have the identities
\begin{multline*}
\int_{\R^3}|u|^2\,\D x=\int_{\R^3}|\chi_R\,u|^2\,\D x+\int_{\R^3}|\xi_R\,u|^2\,\D x\\
\mbox{and}\quad\int_{\R^3}V\,|u|^2\,\D x=\int_{\R^3}V\,|\chi_R\,u|^2\,\D x+\int_{\R^3}V\,|\xi_R\,u|^2\,\D x\;,
\end{multline*}
and the IMS truncation identity
\begin{equation}\label{IMS}
\int_{\R^3}| \nabla(\chi_R\,u)|^2\,\D x+\int_{\R^3}| \nabla(\xi_R\,u)|^2\,\D x=\int_{\R^3}|\nabla u|^2\,\D x-\int_{\R^3}|u|^2\,\underbrace{\nabla\cdot\left(\nabla\chi_R+\nabla\xi_R\right)}_{=O(R^{-2})\;\mbox{as}\; R\to\infty}\,\D x\;.
\end{equation}
A first application of this truncation method is given by the following splitting lemma.
\begin{lemma}\label{Lem:Truncation} For $\rho\in\A_M$, we define $\rho_R^{(1)}=\chi_R\,\rho\,\chi_R$ and $\rho_R^{(2)}=\xi_R\,\rho\,\xi_R$. Then it holds:
\[
\mathcal S[\rho_R^{(1)}]+\mathcal S[\rho_R^{(2)}] \geq \mathcal S[\rho] \quad\mbox{and} \quad \mathcal E_{\rm kin}[\rho_R^{(1)}]+\mathcal E_{\rm kin}[\rho_R^{(2)}]\le\mathcal E_{\rm kin}[\rho]+O(R^{-2})\quad\text{as}\;R\to +\infty\;.
\]
\end{lemma}
\begin{proof} The assertion for $\mathcal E_{\rm kin}[\rho]$ is a straightforward consequence of \eqref{IMS}, namely
\[
\tr(-\Delta\,\rho_R^{(1)})+\tr(-\Delta\,\rho_R^{(2)})=\tr(-\Delta\,\rho)+O(R^{-2})\quad\text{as}\;R\to +\infty\;.
\]
For the entropy term, we can use the \emph{Brown-Kosaki inequality} (cf.~\cite{BrKo}) as in \cite[Lemma 3.4]{DoFeLe} to obtain
\[
\tr\beta(\rho_R^{(1)})+\tr\beta(\rho_R^{(2)}) \le \tr \beta(\rho)\;.
\]
\end{proof}

\subsection{Sub-additivity and maximal temperature}

In order to proceed further, we need to study the dependence of $i_{M,T}$ with respect to $M$ and $T$ and prove that the maximal temperature $T^*$ as defined in \eqref{Tstar} is in fact 
positive. To this end, we rely on the translation invariance of the model. For a given $y\in\R^3$, denote by $\tau_y: L^2(\R^3)\rightarrow L^2(\R^3)$ the translation operator given by
\[
(\tau_yf)=f(\cdot-y)\quad\forall\;f\in L^2(\R^3)\;.
\]
\begin{proposition}\label{sub-additivity} Let $i_{M,T}$ be given by \eqref{infimum} and assume that ($\beta$1)--($\beta$2) hold.  Then the following properties hold:
\begin{itemize}
\item[(i)] As a function of $M$, $i_{M,T}$ is non-positive and sub-additive: for any $M>0$, $m\in(0,M)$ and $T>0$, we have
\[
i_{M,T}\leq i_{M-m,T}+i_{m,T}\le0\;.
\]
\item[(ii)] The function $i_{M,T}$ is a non-increasing function of $M$ and a non-decreasing function of $T$. For any $T>0$, we have $i_{M,T}<0$ if and only if $T<T^*$.
\item[(iii)] For any $M>0$, $T^*(M)> 0$ is positive, possibly infinite. As a function of $M$ it is increasing and satisfies
\[
T^*(M)\geq \max_{0\leq m\leq M}\frac{{m}^3}{\beta(m)}~|i_{1,0}|\;.
\]
As a consequence, $T^*>0$ and $T^*(M)=+\infty$ for any $M>0$ if $\lim_{s\to 0_+}\beta(s)/s^3=0$. \end{itemize}
\end{proposition}
\begin{proof} We start with the proof of the sub-additivity inequality. Consider two states $\rho\in\A_{M-m}$ and $\sigma\in\A_m$, such that $\FT[\rho]\leq i_{M-m,T}+\e$ and $\FT[\sigma]\leq i_{m,T}+\e$. By density of finite rank operators in $\A$ and of smooth compactly supported functions in $L^2$, we can assume that
\[
\rho =\sum_{j=1}^J\lambda_j\,|\psi_j\rangle\langle \psi_j|\;,
\]
with smooth eigenfunctions $(\psi_j)^J_{j=1}$ having compact support in a ball $B(0,R)\subset\R^3$, for some $J\in\N$. After approximating $\sigma$ analogously, we define $\sigma_{R\mathsf e}:=\tau^*_{3R\mathsf e}\,\sigma\,\tau_{3R\mathsf e}$, where $\mathsf e\in\mathbb S^2 \subset\R^3$ is a fixed unit vector and $\tau$ is the translation operator defined above. Note that we have $\rho\,\sigma_{R\mathsf e}=\sigma_{R\mathsf e}\,\rho=0$, hence $\rho+\sigma_{R\mathsf e}\in\A_M$ and $\tr\beta(\rho+\sigma_{R\mathsf e})=\tr\beta(\rho)+\tr\beta(\sigma_{R\mathsf e})$. Thus we have
\[
i_{M,T}\leq \FT[\rho+ \sigma_{R\mathsf e}]=\FT[\rho]+\FT[\sigma]+O(1/R)\leq i_{M-m,T}+i_{m,T}+2\,\e\;,
\]
where the $O(1/R)$ term has in fact negative sign so that we can simply drop it. Taking the limit $\e\rightarrow0$ yields the desired inequality.

Next, consider a minimizer $\rho$ of $\EH$ subject to $\tr\rho=M$. It is given by an appropriate rescaling of the pure state obtained in \cite{Lie}. For an arbitrary $\lambda\in(0,\infty)$, let $(U_\lambda\,f)(x):=\lambda^{3/2} f(\lambda\,x)$ and observe that $\rho_\lambda:=U_\lambda^*\,\rho\,U_\lambda\in\A_M$. As a function of $\lambda$, the Hartree energy $\EH[\rho_\lambda]=\lambda^2\,\mathcal E_{\rm kin}[\rho] - \lambda\,\mathcal E_{\rm pot}[\rho]$ has a minimum for some $\lambda>0$. Computing $\frac{\D}{\D \lambda}\,\EH[\rho_\lambda] =0$, we infer that $\lambda=\mathcal E_{\rm pot}[\rho]/(2\,\mathcal E_{\rm kin}[\rho])$ and moreover
\[
i_{M,0}\equiv \EH[\rho]=-\frac{1}{4}\frac{(\Q[\rho])^2}{\mathcal E_{\rm kin}[\rho]}\;.
\]
As a consequence, we have $i_{M,0}=M^3~i_{1,0}$ and\begin{equation}\label{freeenergypure}
\FT[\rho]=i_{M,0}+T\,\beta(M)=\beta(M)\,\left(T-\frac{M^3}{\beta(M)}|i_{1,0}|\right)\ge i_{M,T}\;,
\end{equation}
thus proving that $i_{M,T}<0$ for $T$ small enough.

Since $\beta$ is non-negative function on $[0,\infty)$, the map $T\mapsto \FT[\rho]$ is increasing. By taking the infimum over all admissible $\rho\in\A_M$, we infer that $T\mapsto i_{M,T}$ is non-decreasing. The function $M\mapsto i_{M,T}$ is non-increasing as a consequence of the sub-additivity property. As a consequence, $T^*(M)$ is a non-decreasing function of $M$, such that
\[
T^*(M)\ge\lim_{M\rightarrow 0_+} T^*(M)\;.
\]
By the sub-additivity inequality and \eqref{freeenergypure}, we obtain
\[
i_{M,T}\leq n\;i_{M/n,T} \leq n\,\beta\left(\tfrac Mn\right)T-\frac{M^3}{n^2}~|i_{1,0}|=n\,\beta\left(\tfrac Mn\right)\left(T-\frac{M^3}{n^3\,\beta\left(\tfrac Mn\right)}~|i_{1,0}|\right)
\]
for any $n\in\N^*$. Since $\lim_{s\rightarrow0_+}\beta(s)/s=0$, we find that $i_{M,T}\leq 0$ by passing to the limit as $n\to\infty$. In the particular case $\lim_{s\rightarrow0_+}\beta(s)/s^3=0$, we conclude that $T^*(M)=+\infty$ for any $M>0$. Similarly, using again the sub-additivity inequality and \eqref{freeenergypure}, we infer
\[
i_{M,T}\leq i_{m,T}\leq \beta(m)\left(T-\frac{{m}^3}{\beta(m)}\,|i_{1,0}|\right)\quad\forall\;m\in(0,M]\;,
\]
which provides the lower bound on $T^*(M)$ in assertion (iii). By definition of $T^*(M)$, we also know that $i_{M,T}$ is negative for any $T<T^*(M)$. From the monotonicity of $T\mapsto i_{M,T}$, we obtain that $i_{M,T}=0$ if $T>T^*$ and $T^*<\infty$. Because of the estimate $i_{M,T}\le i_{M,T_0}+(T-T_0)\,\beta(M)$ for any $T>T_0$, we also find that $i_{M,T^*}=0$ if $T^*<\infty$.\end{proof}

\subsection{Euler-Lagrange equations and Lagrange multipliers} \label{sec:euler}

As in \cite{DoFeMa,DoFeLe}, we obtain the following characterization of $\rho\in\mathfrak M_M$.
\begin{proposition}\label{E-L} Let $M>0$, $T\in(0,T^*(M)]$ and assume that ($\beta$1)--($\beta$2) hold. Consider a density matrix operator $\rho\in\A_M$ which minimizes $\FT$. Then $\rho$ is such that
\begin{equation}\label{Identity:virial}
\tr(V_\rho\,\rho)=4\tr(-\Delta\,\rho)
\end{equation}
and satisfies the self-consistent equation
\begin{equation}\label{SCequation}
\rho=(\beta')^{-1}\big((\mu -H_\rho)/T\big)\,,
\end{equation}
where $H_\rho$ is the mean-field Hamiltonian defined in \eqref{Hamil} and $\mu \le 0$ denotes the Lagrange multiplier associated to the mass constraint $\tr \rho = M$. Explicitly, $\mu$ is given by
\begin{equation}\label{Identity:LagrangeMultiplier}
\mu =\frac1M\,\tr\left((H_\rho+T\,\beta'(\rho))\,\rho\right)\,.
\end{equation}\end{proposition}

\begin{proof} Let $\rho\in\mathfrak M_M$ be a minimizer of $\FT$. Consider the decomposition given by \eqref{decom}. If we denote by $\rho_\lambda$ the density operator in $\A_M$ given by
\[
\rho_\lambda=\lambda^3\sum_{j\in\N}\lambda_j\,|\psi_j(\lambda\cdot)\rangle\langle \psi_j(\lambda\cdot)|\;,
\]
then, as in the proof of Proposition~\ref{sub-additivity}, we find that $\EH[\rho_\lambda]=\lambda^2\,\mathcal E_{\rm kin}[\rho]-\lambda\,\mathcal E_{\rm pot}[\rho]$ while $\s[\rho_\lambda]=\s[\rho]$ for any $\lambda>0$. Hence the condition $\frac d{d\lambda}\,\EH[\rho_\lambda]_{|\lambda=1}=0$ exactly amounts to $\mathcal E_{\rm pot}[\rho]=2\,\mathcal E_{\rm kin}[\rho]$. Next, let $\sigma\in\A_M$. Then $(1-t)\,\rho+t\,\sigma\in\A_M$ and
\[
t\mapsto \FT[(1-t)\,\rho+t\,\sigma ]
\]
has a minimum at $t=0$. Computing its derivative at $t=0$ and arguing by contradiction implies that $\rho$ also solves the linearized problem
\[
\inf_{\sigma\in\A_M}\tr\left((H_\rho +T\,\beta'(\rho))(\sigma-\rho)\right)\,.
\]
Computing the corresponding Euler-Lagrange equations shows that the minimizer of this problem is $\rho = (\beta')^{-1}\big((\mu-H_\rho)/T\big)$ where $\mu$ denotes the Lagrange multiplier associated to the constraint $\tr \rho = M$.
Since the essential spectrum of $H_\rho$ is $[0,\infty)$, we also get that $\mu \le 0$ since $\rho$ is trace class and $(\beta')^{-1}>0$ on $(0,\infty)$. \end{proof}

Using the decomposition \eqref{decom} we can rewrite the stationary Hartree model in terms of (at most) countably many eigenvalue problems coupled through a nonlinear Poisson equation
\[\label{stationary}
\left\{\begin{array}{l}
\Delta\,\psi_j + V_\rho\,\psi_j + \mu_j\,\psi_j = 0\;,\quad j\in\N\;,\\[6pt]
-\Delta V_\rho\,=\,4\pi \sum_{j\in\N}\lambda_j\,|\psi_j|^2\,,
\end{array}\right.\]
where $(\mu_j)_{j\in\N}\in\R$ denotes the sequence of the eigenvalues of $H_\rho$ and $\langle\psi_j,\psi_k\rangle_{L^2}=\delta_{j,k}$. 
The self-consistent equation \eqref{SCequation} consequently implies the following relation between the occupation numbers $(\lambda_j)_{j\in\N}$ and the eigenvalues $(\mu_j)_{j\in\N}$:
\begin{equation}\label{lambda}
\lambda_j=(\beta')^{-1}\left((\mu-\mu_j)/T\right)_+ ,
\end{equation}
where $s_+=(s+|s|)/2$ denotes the positive part of $s$. Upon reverting the relation \eqref{lambda} we obtain $\mu_j=\mu-T\,\beta'(\lambda_j)$ for any $\mu_j\le\mu$.  

\medskip The Lagrange multiplier $\mu$ is usually referred to as the \emph{chemical potential}. In the existence proof given below, it will be essential, that $\mu <0$. In order to show that this is indeed the case, 
let $p(M):=\sup_{m\in(0,M]}\frac{m\,\beta'(m)}{\beta(m)}$. If $\rho\in\A_M$, then
\[
\tr(\beta'(\rho)\,\rho)\le p(M)\,\tr\beta(\rho)\;.
\]
Notice that if ($\beta$3) holds, then $p(M)\le 3$.
\begin{lemma}\label{pLessThan3} Let $M>0$ and $T<T^*(M)$. Assume that $\rho\in\A_M$ is a minimizer of $\FT$ and let $\mu$ be the corresponding Lagrange multiplier. 
With the above notations, if $p(M)\le 3$, then $M\,\mu\le p(M)\,i_{M,T}<0$.\end{lemma}
\begin{proof} By definition of $i_{M,T}$ and according to \eqref{Identity:LagrangeMultiplier}, we know that
\begin{eqnarray*}
&&i_{M,T}=\tr\left(-\Delta\,\rho-\tfrac 12\,V_\rho\,\rho+T\,\beta(\rho)\right)\;,\\
&&M\,\mu=\tr\left(-\Delta\,\rho-V_\rho\,\rho+T\,\beta'(\rho)\,\rho\right)\;.
\end{eqnarray*}
Using \eqref{Identity:virial}, we end up with the identity
\[
p(M)\,i_{M,T}-M\,\mu=(3-p(M))\,\tr(-\Delta\,\rho)+T\,\tr\left(p(M)\,\beta(\rho)-\beta'(\rho)\,\rho\right)\ge 0\;,
\]
which concludes the proof.\end{proof}

The negativity of the Lagrange multiplier $\mu$, is straightforward in the zero temperature case. In our situation it holds under Assumption ($\beta$3), 
but has not been established for instance for $\beta(s)=s^p$ with $p>3$. In fact, it might even be false in some cases, see Section~\ref{conclusion} for more details.

\begin{corollary}\label{StrictMonotonicity} Let $T>0$. Then $M\mapsto i_{M,T}$ is monotone decreasing as long as $T<T^*(M)$ and $p(M)\le3$.\end{corollary}
\begin{proof} 
Let $\rho\in\A_M$ be such that $\FT[\rho]\le i_{M,T}+\varepsilon$, for some $\varepsilon>0$ to be chosen. With no restriction, we can assume that 
$\mathcal E_{\rm pot}[\rho]=2\,\mathcal E_{\rm kin}[\rho]$ and define $\mu[\rho]:=\frac d{d\lambda}\,\FT[\lambda\,\rho]_{ | \lambda=1}$. The same computation as in the proof of Lemma~\ref{pLessThan3} shows that
\[
p(M)\,(i_{M,T}+\varepsilon)-M\,\mu\ge(3-p(M))\,\tr(-\Delta\,\rho)+T\,\tr\left(p(M)\,\beta(\rho)-\beta'(\rho)\,\rho\right)\ge 0\;,
\]
since, by assumption, $p(M)\le 3$. This proves that $M\,\mu[\rho]<i_{M,T}/2<0$ for any $\varepsilon\in(0,|i_{M,T}|/2)$, if $p(M)\le 3$. This bound being uniform with respect to $\rho$, monotonicity easily follows.\end{proof}

\begin{remark} Under the assumptions of Lemma~\ref{pLessThan3}, we observe that $$\frac d{d\lambda}\,\FT[\lambda\,\rho]_{|\lambda=1}=\mu\,M<0,$$ provided $p(M)\le 3$ and $\rho \in \mathfrak H_M$, 
which proves the strict monotonicity of $M\mapsto i_{M,T}$. However, at this stage, the existence of a minimizer is not granted and we thus had to argue differently.\end{remark}

\section{The binding inequality}\label{sec:a-priori}

In this section we shall strengthen the result of Proposition \ref{sub-additivity} (i) and infer a \emph{strict} sub-additivity property of $i_{M,T}$, which is usually called the \emph{binding inequality}; see e.g.~\cite{LeLe}. This will appear as a consequence of the following a priori estimate for the spatial density of the minimizers.
\begin{proposition}\label{decayestimate} Let $\rho\in\A_M$ be a minimizer of $\FT$. There exists a positive constant~$C$ such that, for all $R>0$ sufficiently large,
\[\label{decay}
\int_{|x|>R}n_\rho(x)\;\D x \leq \frac{C}{R^2}\;.
\]\end{proposition}
This result is the analog of \cite[Lemma 5.2]{LeLe}. For completeness, we shall give the details of the proof, which requires $\mu < 0$, in the appendix. The following elementary estimate will be useful in the sequel.
\begin{lemma}\label{Lem:UnifPot} There exists a positive constant~$C$ such that, for any $\rho\in\A_M$,
\[
\int_{\R^3}\frac{n_\rho(x)}{|x|}\;\D x\le C\,M^{3/2}\,\left(\tr(-\Delta\,\rho)\right)^{1/2}\,.
\]\end{lemma}
\begin{proof} Up to a translation, we have to estimate $\int_{\R^3}|x|^{-1}\,n_\rho(x)\,\D x$ and it is convenient to split the integral into two integrals corresponding to $|x|\le R$ and $|x|>R$. By H\"older's inequality, we know that, for any $p>3/2$,
\[
\int_{B_R}\frac{n_\rho(x)}{|x|}\;\D x\le \left(4\pi\,\tfrac{p-1}{2p-3}\right)^{(p-1)/p}\,\norm{n_\rho}_{L^p}\,R^\frac{2p-3}{p-1}\,,
\]
where $B_R$ denotes the centered ball of radius $R$. Similarly, for any $p<3/2$,
\[
\int_{B_R^c}\frac{n_\rho(x)}{|x|}\;\D x\le \left(4\pi\,\tfrac{p-1}{3-2p}\right)^{(p-1)/p}\,\norm{n_\rho}_{L^p}\,R^{-\frac{2p-3}{p-1}}\,.
\]
Applying these two estimates with, for instance, $p=3$ and $p=6/5$ and optimizing w.r.t.~$R>0$, we obtain a limiting case for the Hardy-Littlewood-Sobolev inequalities after using again H\"older's inequality to estimate $\norm{n_\rho}_{L^{6/5}}$ in terms of $\norm{n_\rho}_{L^1}$ and $\norm{n_\rho}_{L^3}$:
\[
\int_{\R^3}\frac{n_\rho(x)}{|x|}\;\D x\le C\,\norm{n_\rho}_{L^1}^{3/2}\,\norm{n_\rho}_{L^3}^{1/2}\,.
\]
We conclude as in \eqref{potentialestimate} using Sobolev's inequality to control $\norm{n_\rho}_{L^3}$ by $\tr(-\Delta\,\rho)$.\end{proof}

As a consequence of Proposition~\ref{decayestimate} and Lemma~\ref{Lem:UnifPot}, we obtain the following result.
\begin{corollary}[Binding inequality] \label{lbindinginequality} Let $M^{(1)}>0$ and $M^{(2)}>0$. If there are minimizers for $i_{M^{(1)},T}$ and $i_{M^{(2)},T}$, then
\[\label{bindinginequality}
i_{M^{(1)}+M^{(2)},T}<i_{M^{(1)},T}+i_{M^{(2)},T}\;.
\]
\end{corollary}
\begin{proof} Consider two minimizers $\rho^{(1)}$ and $\rho^{(2)}$ for $i_{M^{(1)},T}$ and $i_{M^{(2)},T}$ respectively and let $\chi_R$ be the cut-off function given in \eqref{cutoff}. By Lemma~\ref{Lem:Truncation} we have
\[
\tr(-\Delta\,(\chi_R\,\rho^{(\ell)}\,\chi_R))\le\tr(-\Delta\,\rho^{(\ell)})+O(R^{-2})\quad\mbox{and}\quad\tr\beta(\chi_R\,\rho^{(\ell)}\,\chi_R)\le\tr\beta(\rho^{(\ell)})\;.
\]
To handle the potential energies, we observe that
\begin{multline*}
\left|\;\Q[\chi_R\,\rho^{(\ell)}\,\chi_R]-\Q[\rho^{(\ell)}]\;\right|\leq\iint_{\R^3\times\R^3}\frac{(1-\chi_R^2(x)\,\chi_R^2(y))\,n_{\rho^{(\ell)}}(x)\,n_{\rho^{(\ell)}}(y)}{|x-y|}\;\D x\,\D y\\
\leq\iint_{\{|x|\geq R\}\times\{|y|\geq R\}}\frac{n_{\rho^{(\ell)}}(x)\,n_{\rho^{(\ell)}}(y)}{|x-y|}\;\D x\,\D y\;.
\end{multline*}
Using Lemma \ref{decayestimate} and Lemma~\ref{Lem:UnifPot}, we obtain
\[
\left|\;\Q[\chi_R\,\rho^{(\ell)}\,\chi_R]-\Q[\rho^{(\ell)}]\;\right|\leq C\,\left[\tr(-\Delta\, \rho^{(\ell)})\right]^{1/2}\int_{|x|\geq R}n_{\rho^{(\ell)}}(x)\;\D x\leq O(R^{-2})
\]
for $R>0$ large enough. This shows that, for any $R>0$ sufficiently large
\[
\FT[\chi_R\,\rho^{(\ell)}\,\chi_R]\le i_{M^{(\ell)},T}+O(R^{-2})\quad \mbox{ for}\;\ell=1,2.
\]
Consider now the test state
\[
\rho_R:=\chi_R\,\rho^{(1)}\,\chi_R+\tau_{5R\mathsf e}^*\,\chi_R\,\rho^{(2)}\,\chi_R\,\tau_{5R\mathsf e}
\]
for some unit vector $\mathsf e\in\mathbb S^2$. Since $\norm{n_{\rho_R}}_{L^1}\le M^{(1)}+M^{(2)}$, by monotonicity of $M\mapsto i_{M,T}$ (see Proposition~\ref{sub-additivity} (ii)), we get
\begin{multline*}
i_{M^{(1)}+M^{(2)},T}\le\FT[\rho_R]\le\FT[\chi_R\,\rho^{(1)}\,\chi_R]+\FT[\chi_R\,\rho^{(2)}\,\chi_R]-\frac{M^{(1)}M^{(2)}}{9\,R}\\\le i_{M^{(1)},T}+i_{M^{(2)},T}+\frac{C}{R^2} -\frac{M^{(1)}M^{(2)}}{9\,R}
\end{multline*}
for some positive constant $C$, which yields the desired result for $R$ sufficiently large.
\end{proof}

\section{Existence of minimizers below $T^*$}\label {sec:exist}


By a classical result, see e.g.~\cite[Corollary 4.1]{LeLe}, conservation of mass along a weakly convergent minimizing sequence implies that the sequence strongly converges. More precisely, we have the following statement.
\begin{lemma}\label{strong} Let $(\rho_k)_{k\in\N}\in\A_M$ be a minimizing sequence for~$\FT$, such that $\rho_k \rightharpoonup\rho$ weak$-\ast$ in $\A$ and $n_{\rho_k} \to n_\rho$ almost everywhere as $k\to\infty$. Then $\rho_k \rightarrow\rho$ strongly in $\A$ if and only if $\tr\rho=M$.\end{lemma}
\begin{proof} The proof relies on a characterization of the compactness due to Brezis and Lieb (see~\cite{Brezis-Lieb} and \cite[Theorem 1.9]{Lieb-Loss}) from which it follows that
\begin{multline*}
\lim_{k\to\infty}\left(\int_{\R^3}n_{\rho_k}\;\D x -\int_{\R^3}|n_\rho-n_{\rho_k}|\;\D x\right)=\int_{\R^3}n_\rho\;\D x\\
\mbox{and}\quad\lim_{k\to\infty}\Big(\tr(-\Delta\, \rho)-\tr\big(-\Delta\,(\rho-\rho_k)\big)\Big)=\tr(-\Delta\, \rho)\;.
\end{multline*}
By semi-continuity of $\FT$, monotonicity of $M\mapsto i_{M,T}$ according to Proposition~\ref{sub-additivity} (ii) and compactness of the quadratic term in $\EH$, we conclude that $\lim_{k\to\infty}\tr(-\Delta\,(\rho-\rho_k))=0$ if and only if $\tr\rho=M$.\end{proof}


With the results of Section~\ref{sec:basic} in hand, we can now state an existence result for minimizers of $\FT$. To this end, consider a minimizing sequence $(\rho_n)_{n\in\N}$ for $\FT$ and recall that $(\rho_n)_{n\in\N}$ is said to be \emph{relatively compact up to translations} if there is a sequence $(a_n)_{n\in\N}$ of points in~$\R^3$ such that $\tau_{a_n}^*\,\rho_n\,\tau_{a_n}$ strongly converges as $n\to\infty$, up to the extraction of subsequences.

Clearly, the sub-additivity inequality given in Lemma \ref{sub-additivity} (i) is not sufficient to prove the compactness up to translations for $(\rho_n)_{n\in\N}$. More precisely, if \emph{equality} holds, then, as in the proof of Lemma \ref{sub-additivity}, one can construct a minimizing sequence that is \emph{not} relatively compact in $\A$ up to translations. This obstruction is usually referred to as \emph{dichotomy}, cf.~\cite{Li}. To overcome this difficulty, we shall rely on the strict sub-additivity of Corollary~\ref{lbindinginequality}, which, however, only holds for minimizers. This is the main difference with previous works on Hartree-Fock models. As we shall see, the main issue will therefore be to prove the convergence of two subsequences towards minimizers of mass smaller than $M$. 

\begin{proposition} \label{theoexistence} Assume that ($\beta$1)--($\beta$3) hold. Let $M>0$ and consider $T^*=T^*(M)$ defined by \eqref{Tstar}. For all \hbox{$T<T^*$}, there exists an operator $\rho$ in $\A_M$ such that $\FT[\rho]=i_{M,T}$. Moreover, every minimizing sequence $(\rho_n)_{n\in\N}$ for $i_{M,T}$ is relatively compact in $\A$ up to translations. \end{proposition}
\begin{proof} The proof is based on the concentration-compactness method as in \cite{LeLe}. Compared to previous results (see for instance \cite{MR636734,MR992653,MR956083,LeLe}), the main difficulty arises in the splitting case, as we shall see below.

\medskip\noindent \emph{Step 1: Non-vanishing.\/} We split
\[
\Q[\rho_n]=\iint_{\mathbb R^6}\frac{n_{\rho_n}(x)\,n_{\rho_n}(y)}{|x-y|}\,\,\D x \, \D y
\]
into three integrals $I_1$, $I_2$ and $I_3$ corresponding respectively to the domains $|x-y|<1/R$, $1/R<|x-y|<R$ and $|x-y|>R$, for some $R>1$ to be fixed later. Since $n_{\rho_n}$ is bounded in $L^1(\R^3)\cap L^3\subset L^{7/5}(\R^3)$ by Lemma~\ref{Lem:Bound}, by Young's inequality we can estimate $I_1$ by
\begin{align*}
I_1 \leq\norm{n_{\rho_n}}_{L^{7/5}}^2\;\norm{\,|\cdot|^{-1}}_{L^{7/4}(B_{1/R})}\leq \frac C{R^{5/7}}\;,
\end{align*}
and directly get bounds on $I_2$ and $I_3$ by computing
\begin{align*}
&I_2\leq R\iint_{|x-y|<R}n_{\rho_n}(x)\,n_{\rho_n}(y)\;\D x \, \D y\leq R\,M\,\sup_{y\in\mathbb R^3}\int_{y+B_R}n_{\rho_n}(x)\;\D x\;,\\
&I_3\leq \frac{1}{R}\iint_{\R^6} n_{\rho_n}(x)\,n_{\rho_n}(y)\;\D x\,\D y\leq\frac{M^2}R\;.
\end{align*}
Keeping in mind that $i_{M,T}<0$, we have
\[
\FT[\rho_n]\geq i_{M,T}>-I_1-I_2-I_3
\]
for any $n$ large enough, which proves the \emph{non-vanishing} property:
\[
\lim_{n\to\infty}\int_{a_n+B_R}n_{\rho_n}(x)\;\D x\geq \frac 1{R\,M}\,\left(-\,i_{M,T}- \frac{M^2}R-\frac C{R^{5/7}} \right)>0
\]
for $R$ big enough and for some sequence $(a_n)_{n\in\N}$ of points in $\R^3$. Replacing $\rho_n$ by $\tau_{a_n}^*\,\rho_n\,\tau_{a_n}$ and denoting by~$\rho^{(1)}$ the weak limit of $(\rho_n)_{n\in\N}$ (up to the extraction of a subsequence), we have proved that $M^{(1)}=\int_{\R^3}n_{\rho^{(1)}}\,\D x>0$.

\medskip\noindent \emph{Step 2: Dichotomy.\/} Either $M^{(1)}=M$ and $\rho_n$ strongly converges to $\rho$ in $\A$ by Lemma \ref{strong}, or $M^{(1)}\in(0,M)$. Let us choose $R_n$ such that $\int_{\R^3}n_{\rho_n^{(1)}}\,\D x=M^{(1)}+(M-M^{(1)})/n$ where $\rho_n^{(1)}:=\chi_{R_n}\,\rho_n\,\chi_{R_n}$. Let $\rho_n^{(2)}:=\xi_{R_n}\,\rho_n\,\xi_{R_n}$. By definition of $R_n$, $\lim_{n\to\infty}R_n=\infty$. By Step 1, we know that $\rho_n^{(1)}$ strongly converges to $\rho^{(1)}$. By Identity \eqref{IMS} and Lemma~\ref{Lem:Truncation}, we find that
\[
\FT[\rho_n]\ge\FT[\rho_n^{(1)}]+\FT[\rho_n^{(2)}]+O(R_n^{-2})-\iint_{\R^3\times\R^3}\frac{n_{\rho_n^{(1)}}(x)\,n_{\rho_n^{(2)}}(y)}{|x-y|}\;\D x\,\D y\;,
\]
thus showing that
\[
i_{M,T}=\lim_{n\to\infty}\FT[\rho_n]\ge\FT[\rho^{(1)}]+\lim_{n\to\infty}\FT[\rho_n^{(2)}]\;.
\]
By step 1, $\lim_{n\to\infty}\int_{\R^3}n_{\rho_n^{(2)}}\,\D x=M-M^{(1)}$. By sub-additivity, according to Lemma~\ref{sub-additivity} (i), $\rho^{(1)}$ is a minimizer for $i_{M^{(1)},T}$, $(\rho_n^{(2)})_{n\in\N}$ is a minimizing sequence for $i_{M-M^{(1)},T}$ and
\[
i_{M,T}=i_{M^{(1)},T}+i_{M-M^{(1)},T}\;.
\]
Either $i_{M-M^{(1)},T}=0$ and then $i_{M,T}=i_{M-M^{(1)},T}$, which contradicts Corollary~\ref{StrictMonotonicity}, and the assumption $T<T^*$, or $i_{M-M^{(1)},T}<0$. In this case, we can reapply the previous analysis to $(\rho_n^{(2)})_{n\in\N}$ and get that for some $M^{(2)}>0$, $(\rho_n^{(2)})_{n\in\N}$ converges up to a translation to a minimizer $\rho^{(2)}$ for $i_{M^{(2)},T}$ and
\[
i_{M,T}=i_{M^{(1)},T}+i_{M^{(2)},T}+i_{M-M^{(1)}-M^{(2)},T}\;.
\]
{}From Corollary~\ref{lbindinginequality} and~\ref{sub-additivity} (i), we get respectively $i_{M^{(1)}+M^{(2)},T}<i_{M^{(1)},T}+i_{M^{(2)},T}$ and $i_{M^{(1)}+M^{(2)},T}+i_{M-M^{(1)}-M^{(2)},T}\le i_{M,T}$, a contradiction.
\end{proof}


As a direct consequence of the variational approach, the set of minimizers $\mathfrak M_M$ is \emph{orbitally stable} under the dynamics of \eqref{equationrho}. To quantify this stability, define
\[
{\rm dist}_{{\mathfrak M}_M}(\sigma):=\inf_{\rho\in\mathfrak M_M}\norm{\rho-\sigma}_{\A}\;.
\]
\begin{corollary}\label{Corstab} Assume that ($\beta$1)--($\beta$3) hold. For any given $M>0$, let $T\in(0,T^*(M))$. For any $\e>0$, there exists $\delta>0$ such that, for all $\rho_{\rm in}\in\A_M$ with ${\rm dist}_{{\mathfrak M}_M}(\rho_{\rm in})\leq \delta$,
\[
\sup_{t\in\R_+}{\rm dist}_{{\mathfrak M}_M}(\rho(t))\leq \e
\]
where $\rho(t)$ is the solution of \eqref{equationrho} with initial data $\rho_{\rm in}\in\A_M$.\end{corollary}
Similar results have been established in many earlier papers like, for instance in \cite{MaReWo} in the case of repulsive Coulomb interactions. As in \cite{CaLi, MaReWo}, the result is a direct consequence of the conservation of the free energy along the flow and the compactness of all minimizing sequences. According to \cite{Lie}, for $T\in(0,T_c]$, the minimizer corresponding to $i_{M,T}$ is unique up to translations (see next Section). A much stronger stability result can easily be achieved. Details are left to the reader.

\section{Critical Temperature for mixed states}\label{sec:critical}

In this subsection, we shall deduce the existence a critical temperature $T_c\in (0, T^*)$, above which minimizers $\rho\in\mathfrak M_M$ become true mixed states, i.e.~density matrix operators with rank higher than one.
\begin{lemma}\label{concavity} For all $M>0$, the map $T\mapsto i_{M,T}$ is concave.\end{lemma}
\begin{proof} Fix some $T_0>0$ and write
\[
\FT[\rho]=\mathcal F_{T_0}[\rho]+(T-T_0)\,|\mathcal S[\rho]|\;.
\]
Denoting by $\rho_{T_0}$ the minimizer for $\mathcal F_{T_0}$, we obtain
\[
i_{M,T}\leq i_{M,T_0}+(T-T_0)\,|\mathcal S[\rho_{T_0}]|
\]
which means that $|\mathcal S[ \rho_{T_0}]|$ lies in the cone tangent to $T\mapsto i_{M,T}$ and $i_{M,T}$ lies below it, i.e.~$T\mapsto i_{M,T}$ is concave.\end{proof}

Consider $T_c$ defined by \eqref{criticalT}, i.e.~the largest possible $T_c$ such that $i_{M,T}=i_{M,0}+T\,\beta(M)$ for $T\in[0,T_c]$ and recall some results concerning the zero temperature case. Lieb in \cite{Lie} proved that $\mathcal F_{T=0}=\EH$ has a unique radial minimizer $\rho_0 = M\,|\psi_0\rangle\langle \psi_0|$. The corresponding Hamiltonian operator
\begin{equation}\label{Hzero}
H_0:= - \Delta - |\psi_0|^2 \ast |\cdot|^{-1}=H_{\rho_0}
\end{equation}
admits countably many negative eigenvalues $(\mu^0_j)_{j\in\N}$, which accumulate at zero. We shall use these eigenvalues to characterize the critical temperature $T_c$. 
To this end we need the following lemma.
\begin{lemma}\label{exofTc} Assume that ($\beta$1)--($\beta$3) hold.  With $T_c$ defined by \eqref{criticalT}, $T_c(M)$ is positive for any $M>0$.\end{lemma}
\begin{proof} Consider a sequence $(T_n)_{n\in\N}\in\R_+$ such that $\lim_{\to\infty}T_n=0$. Let $\rho^{(n)}\in\mathfrak M_M$ denote the associated sequence of minimizers with occupation numbers $(\lambda_{j}^{(n)})_{j\in\N}$. According to~\eqref{lambda}, we know that
\[
\lambda_{j}^{(n)}=(\beta')^{-1}\left((\mu^{(n)}-\mu_{j}^{(n)})/T_n\right)\quad\forall\;j\in\N\;,
\]
where, for any $n\in\N$, $\big(\mu^{(n)}_j\big)_{j\in\N}$ denotes the sequence of eigenvalues of $H_{\rho^{(n)}}$ and $\mu^{(n)}\le 0$ is the associated chemical potential. Since $\rho^{(n)}$ is a minimizing sequence for $\mathcal F_{T=0}$, we know that
\[
\lim_{n\to\infty}
\mu_j^{(n)}=\mu_j^0\leq 0
\]
where $(\mu_j^0)_{j\in\N}$ are the eigenvalues of $H_0$. Arguing by contradiction, we assume that
\[
\liminf_{n\to \infty}\lambda_1^{(n)}=\epsilon>0\;.
\]
By \eqref{lambda} and the fact that $\beta'$ is increasing, this implies: $\mu^{(n)}>\mu_1^{(n)}\to\mu_1^0$ as $n\to\infty$. Then
\[
M=\lambda_0^0\geq\lim_{\to\infty}\lambda_0^{(n)}=\lim_{\to\infty}(\beta')^{-1}\left(\tfrac{\mu^{(n)}-\mu_{0}^{(n)}}{T_n}\right)\geq \lim_{\to\infty}(\beta')^{-1}\left(\tfrac{\mu_1^0-\mu_{0}^{(n)}}{T_n}\right)=+\infty\;.
\]
This proves that there exists an interval $[0,T_c)$ with $T_c>0$ such that, for any $T_n\in[0,T_c)$, it holds $\mu^{(n)}<\mu^{(n)}_1$, and, as a consequence, $\rho^{(n)}$ is of rank one. Hence, for any $T\in[0,T_c)$, the minimizer of $\FT$ in $\A_M$ is also a minimizer of $\EH+T\,\beta(M)$. From \cite{Lie}, we know that it is unique and given by $\rho_0$, in which case $i_{M,T}=i_{M,0}- T\,\mathcal S[\rho_0]=i_{M,0}+ T\,\beta [M]$.\end{proof}

As an immediate consequence of Lemmata~\ref{concavity} and~\ref{exofTc} we obtain the following corollary.
\begin{corollary} Assume that ($\beta$1)--($\beta$3) hold. There is a pure state minimizer of mass $M$ if and only if $T\in[0,T_c]$.\end{corollary}
\begin{proof} A pure state satisfies $i_{M,T}=i_{M,0}+T\,\beta(M)$ and from the concavity property stated in Lemma \ref{concavity} we conclude $i_{M,T}<i_{M,0}+T\,\beta(M)$ for all $T>T_c$.
\end{proof}
We finally give a characterization of $T_c$.
\begin{proposition} Assume that ($\beta$1)--($\beta$3) hold. For any $M>0$, the critical temperature satisfies
\[\label{Tcrit}
T_c(M) = \, \frac{\mu^0_1-\mu^0_0}{\beta'(M)}\;,
\]
where $\mu^0_0$ and $\mu^0_1$ are the two lowest eigenvalues of $H_0$ defined in \eqref{Hzero}.\end{proposition}
\begin{proof} For $T \leq T_c$, there exists a unique pure state minimizer $\rho_0$. For such a pure state, the Lagrange multiplier associated to the mass constraint $\tr \rho_0 = M$ is given by $\mu= \mu(T)$. According to~\ref{Identity:LagrangeMultiplier}, it is given by $T\,\beta'(M)+\mu_0^0-\mu(T)=0$ for any $T<T_c$ (as long as the minimizer is of rank one). This uniquely determines $\mu(T)$. On the other hand we know that $0\neq\lambda_1=(\beta')^{-1}\left((\mu_1^0-\mu(T))/T\right)$ if $T>(\mu_1^0-\mu_0^0)/\beta'(M)$, thus proving that $T_c\le(\mu^0_1-\mu^0_0)/\beta'(M)$.

\medskip It remains to prove equality: By using Lemmas \ref{concavity} and \ref{exofTc}, we know that $i_{M,T_c}=i_{M,0}+T_c\,\beta(M)$. Let $\rho$ be a minimizer for $T=T_c$. The two inequalities, $i_{M,0}\le\EH[\rho]$ and $\beta(M)\le\tr\beta(\rho)$ hold as equalities if and only if, in both cases, $\rho$ is of rank one. Consider a sequence $(T^{(n)})_{n\in\N}$ such that $\lim_{n\to\infty}T^{(n)}=T_c$, $T^{(n)}>T_c$ for any $n\in\N$ and, if $(\rho^{(n)})_{n\in\N}$ denotes a sequence of associated minimizers with $\big(\mu^{(n)}_j\big)_{j\in\N}$ and \hbox{$\mu^{(n)}\le 0$} as in the proof of Lemma~\ref{exofTc}, we have $\mu^{(n)}>\mu^{(n)}_1$ so that $\lambda^{(n)}_1>0$ for any $n\in\N$. The sequence $(\rho^{(n)})_{n\in\N}$ is minimizing for $i_{M,T_c}$, thus proving that $\lim_{n\to\infty}\lambda^{(n)}_1=0$, so that \hbox{$\lim_{n\to\infty}\mu^{(n)}=\mu_1^0$}. Passing to the limit in
\[
M\,\mu^{(n)}=\sum_{j\in\N}\lambda^{(n)}_j\left(\mu^{(n)}_j+T^{(n)}\,\beta'(\lambda^{(n)}_j)\right)
\]
completes the proof.\end{proof}

\section{Estimates on the maximal temperature}\label{sec:finite}


All above results require $T<T^*$, the maximal temperature. In some situations, we can prove that $T^*$ is finite.
\begin{proposition}\label{propfinite}
Let $\beta(s)=s^p$ with $p\in(1,7/5)$. Then, for any $M>0$, the maximal temperature $T^*=T^*(M)$ is finite. \end{proposition}
\begin{proof} Let $V$ be a given non-negative potential. From \cite{DoFeLoPa}, we know that
\[
2\,T\,\tr\beta(\rho)+\tr(-\Delta\,\rho)-\tr(V\rho)\ge-(2\,T)^{-\frac 1{p-1}}\,(p-1)\,p^{-\frac p{p-1}}\sum_j|\mu_j(V)|^\gamma
\]
where $\gamma=\frac p{p-1}$ and $\mu_j(V)$ denotes the negative eigenvalues of $-\Delta-V$. The sum is extended to all such eigenvalues. By the Lieb-Thirring inequality, we have the estimate
\[
\sum_j|\mu_j(V)|^\gamma \le C_{\rm LT}(\gamma)\int_{\R^3}|V|^q\,\D x
\]
with $q=\gamma+\frac 32$. In summary, this amounts to
\[
2\,T\,\tr\beta(\rho)+\tr(-\Delta\,\rho)-\tr(V\rho)\ge-(2\,T)^{-\frac 1{p-1}}\,(p-1)\,p^{-\frac p{p-1}}\,C_{\rm LT}(\gamma)\int_{\R^3}|V|^q\,\D x\;.
\]
Applying the above inequality to $V=V_\rho = n_\rho \ast |\,\cdot\,|^{-1}$, we find that
\begin{align*}
\FT[\rho]=& \ \frac 12\,\tr(-\Delta\,\rho)+\frac 12\,\Big[(2\,T)\,\tr\beta(\rho)+\tr(-\Delta\,\rho)-\tr(V_\rho\,\rho)\Big]\\
\ge & \ \frac 12\,\tr(-\Delta\,\rho)-T^{-\frac 1{p-1}}\,(2\,p)^{-\frac p{p-1}}\,C_{\rm LT}(\gamma)\int_{\R^3}|V_\rho|^q\,\D x\;.
\end{align*}
Next, we invoke the Hardy-Littlewood-Sobolev inequality
\[
\int_{\R^3}|V_\rho|^q\,\D x\le C_{\rm HLS}\,\|n_\rho\|_{L^r(\R^3)}^q
\]
for some $r>1$ such that $\frac 1r=\frac 23+\frac 1q$. Notice that $r>1$ means $q>3$ and hence $p<3$. H\"older's inequality allows to estimate the right hand side by
\[
\|n_\rho\|_{L^r(\R^3)}\le\|n_\rho\|_{L^1(\R^3)}^\theta\,\|n_\rho\|_{L^3(\R^3)}^{1-\theta}
\]
with $\theta=\frac 32\,\big(\frac 1r-\frac 13\big)$. Since $\|n_\rho\|_{L^3(\R^3)}$ is controlled by $\norm{\nabla\sqrt{n_\rho}}_{L^2}^2$ using Sobolev's embedding, which is itself bounded by $\tr(-\Delta\,\rho)$, we conclude that
\[
\int_{\R^3}|V_\rho|^q\,\D x\le c\,M^{q\,\theta}\,\left(\tr(-\Delta\,\rho)\right)^{q\,(1-\theta)}
\]
for some positive constant $c$ and, as a consequence,
\begin{equation}\label{Estim:LiebThirring}
\FT[\rho]\ge\frac 12\,\tr(-\Delta\,\rho)-T^{-\frac 1{p-1}}\,\mathsf K\tr(-\Delta\,\rho)^{q\,(1-\theta)},
\end{equation}
for some $\mathsf K>0$. Moreover we find that
\[
q\,(1-\theta)=1+\eta\quad\mbox{with}\quad\eta=\frac{7-5\,p}{4\,(p-1)}\;,
\]
so that $\eta$ is positive if $p\in(1,7/5)$.

Assume that $i_{M,T}<0$ and consider an admissible $\rho\in\A_M$ such that $\FT[\rho]=i_{M,T}$. Since $\tr\beta(\rho)$ is positive, as in the proof of \eqref{potentialestimate}, we know that for some positive constant $C$, which is independent of $T>0$,
\[
0>\FT[\rho]>\EH[\rho]\ge\tr(-\Delta\,\rho)-C\,M^{3/2}~\tr(-\Delta\,\rho)^\frac 12\,,
\]
and, as a consequence,
\[
\tr(-\Delta\,\rho)\le C^2\,M^3\,.
\]
On the other hand, by \eqref{Estim:LiebThirring}, we know that $\FT[\rho]<0$ means that
\[
\tr(-\Delta\,\rho)>\left(\frac{T^\frac 1{p-1}}{2\,\mathsf K}\right)^\frac 1\eta.
\]
The compatibility of these two conditions amounts to
\[
T^\frac 1{p-1}\le 2\,\mathsf K\,C^{2\,\eta}\,M^{3\,\eta}\,,
\]
which provides an upper bound for $T^*(M)$.\end{proof}

Finally, we infer the following asymptotic property for the infimum of $\FT[\rho]$.
\begin{lemma} Assume that ($\beta$1)--($\beta$2) hold. If $T^* < +\infty$, then $\lim_{T\rightarrow T^*_-}i_{M,T}=0$.
\end{lemma}
\begin{proof} The proof follows from the concavity of $T\mapsto i_{M,T}$ (see Lemma~\ref{concavity}). Let $\rho_{T_0}$ denote the minimizer at $T_0<T^*$, with $\mathcal F_{T_0}[\rho_{T_0}] = - \delta $ for some $\delta >0$. Then we observe
\[
i_{M,T} \leq (T- T_0) \sum_{j\in\N} \beta(\lambda_j) + \mathcal F_{T_0}[\rho_{T_0}] \leq (T- T_0)\,\beta(M) - \delta < 0\;,
\]
for all $T$ such that: $T-T_0 \le \delta /\beta(M)$, which is in contradiction with the definition of $T^*$ given in \eqref{Tstar} if $\liminf_{T\rightarrow T^*_-}i_{M,T}<0$.\end{proof}

\section{Concluding remarks}\label{conclusion}

Assumption ($\beta$3) is needed for Corollary~\ref{StrictMonotonicity}, which is used itself in the proof of Proposition~\ref{theoexistence} (compactness of minimizing sequences). When $\beta(s)=s^p$, this means that we have to introduce the restriction $p\le3$. If look at the details of the proof, what is really needed is that $\mu=\frac{\partial\,i_{M,T}}{\partial M}$ takes negative values. To further clarify the role of the threshold $p=3$, we can state the following result.
\begin{proposition} Assume that $\beta(s)=s^p$ for some $p>1$. Then we have
\begin{equation}\label{Inequality:PartialDerivatives}
M\,\frac{\partial\,i_{M,T}}{\partial M}+(3-p)\,T\,\frac{\partial\,i_{M,T}}{\partial T}\le3\,i_{M,T}
\end{equation}
and, as a consequence:
\begin{enumerate}
\item[(i)] if $p\le 3$, then $i_{M,T}\le(\tfrac M{M_0})^3\,i_{M,T_0}$ for any $M>M_0>0$ and $T>0$.
\item[(ii)] if $p\ge 3$, then $i_{M,T}\le(\tfrac T{T_0})^{3/(3-p)}\,i_{M,T_0}$ for any $M>0$ and $T>T_0>0$.
\end{enumerate}\end{proposition}
\begin{proof} Let $\rho\in\A_M$ and, using the representation~\eqref{decom}, define
\[
\rho_\lambda:=\lambda^4\sum_{j\in\N}\lambda_j\,|\psi_j(\lambda\cdot)\rangle\langle \psi_j(\lambda\cdot)|.
\]
With $M[\rho]:=\tr\rho=\int_{\R^3}n_\rho\,\D x$, we find that
\[
M[\rho_\lambda]=\lambda\,M[\rho]=\lambda\,M
\]
and
\[
\mathcal F_{\lambda^{3-p}\,T}[\rho_\lambda]=\lambda^3\,\FT[\rho]\;.
\]
As a consequence, we have
\[
i_{\lambda M,\lambda^{3-p}\,T}\le\lambda^3\,i_{M,T}\;,
\]
which proves \eqref{Inequality:PartialDerivatives} by differentiating at $\lambda=1$. In case (i), since $T\mapsto i_{M,T}$ is non-decreasing, we have
\[
i_{\lambda M_0,T}\le i_{\lambda M_0,\lambda^{3-p}\,T}\le\lambda^3\,i_{M_0,T}\quad\forall\;\lambda>1
\]
and the conclusion holds with $\lambda=M/M_0$. In case (ii), since $M\mapsto i_{M,T}$ is non-increasing, we have
\[
i_{M,\lambda^{3-p}\,T_0}\le i_{\lambda M,\lambda^{3-p}\,T_0}\le\lambda^3\,i_{M,T_0}\quad\forall\;\lambda\in(0,1)
\]
and the conclusion holds with $\lambda=(T/T_0)^{1/(3-p)}$.\end{proof}

Assume that $\beta(s)=s^p$ for any $s\in\R^+$. We observe that for $T<T^*(M)$, $\frac{\partial\,i_{M,T}}{\partial M}\le\frac 3M\,i_{M,T}$ if $p\le 3$, but we have no such estimate if $p>3$. In Proposition~\ref{sub-additivity}~(iii), the sufficient condition for showing that $T^*(M)=\infty$ is precisely $p>3$. Hence, at this stage, we do not have an example of a function $\beta$ satisfying Assumptions ($\beta$1) and ($\beta$2) for which existence of a minimizer of $i_{M,T}$ in $\A_M$ is granted for any $M>0$ and any $T>0$. In other words, with $T^*$ can be infinite for a well chosen function $\beta$, for instance $\beta(s)=s^p$, $s\in\R^+$, for $p>3$. However, in such a case we do not know if the Lagrange multiplier $\mu(T)$ is negative for any $T>0$ and as a consequence, the existence of a minimizer corresponding to $i_{M,T}$ is an open question for large values of $T$.

\appendix\section{Proof of Proposition \ref{decayestimate}}

Consider the minimizer $\rho$ of Proposition~\ref{decayestimate} and let $\mu<0$ be the Lagrange multiplier corresponding to the mass constraint $\tr \rho = M$. Define
\[
\mathcal G^\mu_T[\rho] := \FT[\rho] - \mu\,\tr (\rho)\;.
\]
The density operator $\rho$ is a minimizer of the unconstrained minimization problem $\inf_{\rho\in\A} \mathcal G_T^\mu[\rho]$. By the same argument as in the proof of Proposition \ref{E-L} we know that $\rho$ also solves the linearized minimization problem $\inf_{\sigma\in\A}\mathcal L^\mu[\sigma]$ where
\[
\mathcal L^\mu[\sigma]:= \tr\left[\left(H_\rho - \mu+T\,\beta'(\rho)\right)\sigma\right]\,.
\]

Consider the cut-off functions $\chi_R$ and $\xi_R$ defined in~\eqref{cutoff} and let $\rho_R:=\chi_R\,\rho\,\chi_R$. By Lemma~\ref{Lem:Truncation}, we know that, as $R\to\infty$,
\[
\tr(-\Delta\,\rho)\geq \tr(-\Delta\,\rho_R)+ \tr(-\Delta\,(\xi_R\,\rho\,\xi_R))-\frac{C}{R^2}
\]
for some positive constant $C$. Next we rewrite the potential energy as
\begin{multline*}
\Q[\rho]=\iint_{\R^3\times\R^3}\frac{n_\rho(x)\,\chi_R^2(y)\,n_\rho(y)}{|x-y|}\;\D x\,\D y +\iint_{\R^3\times\R^3}\frac{\chi_{R/4}^2(x)\,n_\rho(x)\,\xi_R^2(y)\,n_\rho(y)}{|x-y|}\;\D x\,\D y \\
+\iint_{\R^3\times\R^3}\frac{\xi_{R/4}^2\,(x)\,n_\rho(x)\,\xi_R^2(y)\,n_\rho(y)}{|x-y|}\;\D x\,\D y\;.
\end{multline*}
In the second integral we use the fact that $|x-y|\geq R/2$, whereas the third integral can be estimated by Lemma~\ref{Lem:UnifPot}. Using the fact that
\begin{align*}
\varepsilon(R):=& \, \tr(-\Delta\,(\xi_R\,\rho\,\xi_R))\\ 
& \, =\!\sum_{j\in\N}\lambda_j\!\!\int_{\R^3}|\nabla(\xi_R\,\psi_j)|^2\,\D x 
\le 2\,\frac M{R^2}\,\norm{\nabla\xi}_{L^\infty}^2+2\sum_{j\in\N}\lambda_j\!\!\int_{\R^3}\xi_R^2\,|\nabla\psi_j|^2\,\D x
\end{align*}
converges to $0$ as $R\to\infty$, we obtain that $\norm{\xi_{R/4}^2\,n_\rho\ast|\cdot|^{-1}}_{L^\infty}\le C\,\sqrt{\varepsilon(R/4)}\to 0$ and can estimate the third integral by
\[
\iint_{\R^3\times\R^3}\frac{\xi_{R/4}^2\,(x)\,n_\rho(x)\,\xi_R^2(y)\,n_\rho(y)}{|x-y|}\;\D x\,\D y\le C\,\sqrt{\varepsilon(R/4)}\int_{\R^3}\xi_R^2(y)\,n_\rho(y)\;\D x\;.
\]

In summary this yields
\[
\Q[\rho]\leq\tr(V_\rho\,\rho_R)+ o(1)\int_{\R^3}\xi_R^2\,n_\rho\;\D x\;.
\]

Collecting all estimates, we have proved that

\[
\mathcal L^\mu[\rho_R]\leq \mathcal L^\mu[\rho] - \varepsilon(R) + \left(\mu+o(1)\right)\int_{\R^3}\xi_R^2\,n_\rho\;\D x+\frac{C}{R^2}
\]
as $R\to\infty$. Recall that $\varepsilon(R)$ is non-negative, $\mu$ is negative (by Lemma \ref{pLessThan3}) and $\rho$ is a minimizer of $\mathcal L^\mu$ so that $\mathcal L^\mu[\rho]\le\mathcal L^\mu[\rho_R]$. As a consequence,
\[
\left(\mu+o(1)\right)\int_{\R^3}\xi_R^2\,n_\rho\;\D x+\frac{C}{R^2}\ge 0
\]
for $R$ large enough, which completes the proof of Proposition \ref{decayestimate}.\qed

\bigskip\noindent\emph{Acknowledgments.\/} {\small The authors thank P.~Markowich and G.~Rein for helpful discussions.}

\bibliographystyle{amsplain}

\medskip

{\small \copyright\,2010 by the authors. This paper may be reproduced, in its entirety, for non-commercial purposes.}

\end{document}